\documentclass[a4paper,11pt]{amsart} 
\usepackage{amssymb, amsmath}
\usepackage{amscd}
\numberwithin{equation}{section}
\usepackage{epsfig}
\usepackage{amsmath}
\usepackage{amsfonts,amssymb,amsopn}

\usepackage{amsthm}
\usepackage{hyperref}
\usepackage{verbatim}
\usepackage{color}
\usepackage[color,all]{xy}

\newcommand{\cO}{{\mathcal O}}
\newcommand{\C}{{\mathbb C}}
\newcommand{\PP}{{\mathbb P}}
\newcommand{\bV}{\overline{V}}

\newcommand{\tC}{\widetilde{C}}
\newcommand{\tF}{\widetilde{F}}
\newcommand{\tW}{\widetilde{W}}
\newcommand{\tx}{\widetilde{x}}

\newcommand{\Ker}{\mathrm{Ker}}
\newcommand{\Image}{\mathrm{Im}\,}

\newcommand{\Hom}{\mathrm{Hom}}
\newcommand{\rank}{\mathrm{rk}\,}
\newcommand{\ev}{\mathrm{ev}}

\newcommand{\Kc}{K_{C}}
\newcommand{\Oc}{{{\cO}_{C}}}
\newcommand{\Oz}{{{\cO}_{Z}}}

\newcommand{\Quot}{\mathrm{Quot}}

\newcommand{\codim}{\mathrm{codim}}
\newcommand{\Cliff}{\mathrm{Cliff}}
\newcommand{\reddeg}{\mathrm{red \, deg \,}}

\newcommand{\Gr}{\mathrm{Gr}}

\newcommand{\mev}{M_{V, E}}
\newcommand{\mlv}{M_{V, L}}
\newcommand{\mfw}{M_{W, F}}
\newcommand{\mfv}{M_{V, H^{2n+1}}}
\newcommand{\opo}{\cO_{\PP^1}}
\newcommand{\opeo}{\cO_{\PP E} (1)}

\newtheorem{theorem}{{\textbf Theorem}}[section]
\newtheorem{proposition}[theorem]{{\textbf Proposition}}
\newtheorem{corollary}[theorem]{{\textbf Corollary}}
\newtheorem{lemma}[theorem]{{\textbf Lemma}}
\newtheorem{exit}[theorem]{{\textbf Example}}
\newtheorem{conjecture}[theorem]{{\textbf Conjecture}}
\newtheorem{defn}[theorem]{{\textbf Definition}}
\newtheorem{remit}[theorem]{{\textbf Remark}}
\newtheorem{quit}{{\textbf Question}}
\newenvironment{question}{\begin{quit}\rm}{\end{quit}}
\newenvironment{example}{\begin{exit}\rm}{\end{exit}}
\newenvironment{remark}{\begin{remit}\rm}{\end{remit}}
\newenvironment{definition}{\begin{defn}\rm}{\end{defn}}

\title{Linear stability of coherent systems and applications to Butler's conjecture}

\author{Abel Castorena}

\address{A.\ C.\ and E.\ D.\ L.\ N.: Centro de Ciencias Matem\'aticas -- UNAM Campus Morelia, Antigua Carretera a P\'atzcuaro \# 8701, Col.\ Ex Hacienda San Jos\'e de la Huerta, Morelia, Michoac\'an, Mexico C.\ P.\ 58089.}

\author{George H.\ Hitching}

\address{G.\ H.\ H.: Oslo Metropolitan University, Postboks 4, St. Olavs plass, 0130 Oslo, Norway.}

\author{Erick Luna}

\email{abel@matmor.unam.mx; gehahi@oslomet.no; eluna@matmor.unam.mx}

\subjclass[2010]{14H45; 14H51; 14H60}

\keywords{Coherent system, curve, linear stability, dual span bundle}

\begin{document}

\begin{abstract}
The notion of linear stability of a variety in projective space was introduced by Mumford in the context of GIT. It has subsequently been applied by Mistretta and others to Butler's conjecture on stability of the dual span bundle (DSB) $\mev$ of a general generated coherent system $( E, V )$. We survey recent progress in this direction on rank one coherent systems, prove a new result for hyperelliptic curves, and state some open questions. We then extend the definition of linear stability to generated coherent systems of higher rank. We show that various coherent systems with unstable DSB studied in \cite{bmno} are also linearly unstable. We show that linearly stable coherent systems of type $(2, d, 4)$ for low enough $d$ have stable DSB, and use this to prove a particular case of a generalized Butler conjecture. We then exhibit a linearly stable generated coherent system with unstable DSB, confirming that linear stability of $( E, V )$ in general remains weaker than semistability of $\mev$ in higher rank. We end with a list of open questions on the higher rank case.
\end{abstract}

\maketitle

\section{Introduction}

\noindent In the context of geometric invariant theory, David Mumford introduced in \cite{mumfordstability} the notion of linear stability for projective varieties $X \subset \PP^{n-1}$. This definition in some way measures how $X$ sits in $\PP^{n-1}$. Intuitively, if $X$ is linearly stable then a low-dimensional subspace of $\PP^{n-1}$ should not contain too many points of $X$. The precise formulation \cite[Definition 2.16]{mumfordstability} is:

\begin{definition} \label{MumfordDefn}
Let $X \subset \PP^{n-1} = \PP V^\vee$ be a variety of dimension $r$. The \textsl{reduced degree} of $X$ is defined as
\[ \reddeg ( X ) \ := \ \frac{\deg (X)}{n - \dim ( X )} . \]
Then $X$ is said to be \textsl{linearly semistable} if for all subspaces $W \subset V$ such that the image of the projection $p_W \colon X \dashrightarrow \PP W^\vee$ has dimension $r$, we have
\[ \reddeg \left( p_W ( X ) \right) \ \ge \ \reddeg ( X ) . \]
If inequality is strict for all such $W$, then $X$ is said to be \textsl{linearly stable}.
\end{definition}

\noindent Mumford \cite[Theorem 4.12]{mumfordstability} went on to show that when $X$ is a curve, linear stability implies Chow stability (that is, stability of the Chow form of $X$ in the GIT sense). This was used by Stoppino \cite{Sto} and Barja--Stoppino \cite{BS} to derive slope inequalities for fibred surfaces. Brambila-Paz and Torres-L\'opez also applied it in \cite{bptl} to prove Chow stability of many curves in projective space, and to study Hilbert schemes of curves in $\PP^{n-1}$.

 \vspace{.25cm}
 
\noindent The main focus of the present work is another application of linear stability studied by Mistretta, Stoppino and others. Recall that a \textsl{generated coherent system} over a smooth curve $C$ is a pair $( E, V )$ where $E \to C$ is a vector bundle and $V \subseteq H^0 ( C, E )$ a subspace such that the evaluation map $V \to E|_p$ is surjective for all $p \in C$. To such an $(E, V)$ we can associated a natural exact sequence
\begin{equation}
0 \ \to \ \mev \ \to \ V \otimes \Oc \ \to \ E \ \to \ 0 \label{DefnDSB}
\end{equation}
where the bundle $\mev$ has variously been called the \textsl{dual span bundle} (DSB), \textsl{syzygy bundle}, \textsl{kernel bundle} and \textsl{evaluation bundle} of the pair $(E, V)$. We write $\mev = M_E$ when $V = H^0(C,E)$. For $\mev$ as above, $\left( \mev^\vee, V^\vee \right)$ is also a generated coherent system. There is a notion of stability for coherent systems, depending on a positive real parameter $\alpha$, which is recalled in Definition \ref{alphaSt} below. The well known Butler conjecture \cite[Conjecture 2]{but} can be formulated as follows:

\begin{conjecture} \label{ButlerConjecture}
Suppose that $C$ is a general curve of genus $g \ge 3$ and $(E, V)$ a generated coherent system which is $\alpha$-stable for $\alpha$ close to $0$ and general among such coherent systems. Then $( \mev^\vee , V^\vee )$ is also $\alpha$-stable.
\end{conjecture}

\noindent As discussed in \cite{but} (see {\S} \ref{BackgroundCohSys} below), when $\alpha$ is close to zero the ambient vector bundle of an $\alpha$-stable coherent system is semistable. We shall be interested in the following slightly stronger conjecture.

\begin{conjecture} \label{StrongButlerConjecture}
Suppose that $C$ is any curve. Suppose that $(E, V)$ is a generated coherent system which is $\alpha$-stable for $\alpha$ close to $0$ and general among such coherent systems. Then $\mev$ is a stable vector bundle.
\end{conjecture}

\noindent The reason for extending to arbitrary curves is that we shall later prove a case of Conjecture \ref{StrongButlerConjecture} for bielliptic curves. See also Remark \ref{rmkBielliptic}.

 \vspace{.25cm}
 
\noindent Butler's conjecture is of relevance for higher rank Brill--Noether theory \cite{bgpmn}, syzygies of curves \cite{farkaslarson}, and generalised theta divisors \cite{Popa}. The conjecture has been proven in many cases, especially when $E$ is a line bundle. The literature on the subject is considerable; we refer the reader to the landmark papers \cite{bbn2} and \cite{farkaslarson} for further information, and to \cite{bmno} for recent progress on the case $\rank ( E ) \ge 2$.

 \vspace{.25cm}
 
\noindent Let us explain the connection between Butler's conjecture and linear stability. Let $( L, V )$ be a base point free linear series over a smooth curve $C$ of genus $g \ge 0$; equivalently, a generated coherent system of rank one. Then as observed in \cite[Remark 3.2]{mistrettastoppino} (see also \cite[Definition 2.1]{BS}), linear stability of the image of the natural map $C \to \PP V^\vee$ is equivalent to the statement that $\mlv$ is not destabilized by any subbundle of the form $M_{V', L'}$ for a coherent subsystem $(L' , V')$ of $(L, V)$. (See Definition \ref{LinStRankOne} and Remark \ref{SlopeEquivRankOne} below for a more complete formulation.)

 \vspace{.25cm}
 
\noindent In particular, linear semistability of $( L, V )$ is a weaker condition than (slope) stability of the vector bundle $\mlv$; that is, slope (semi)stability of $\mlv$ implies linear (semi)stability of the pair $( L, V )$. As linear (semi)stability can be easier to check, with Butler's conjecture in mind it is of interest to determine conditions under which the reverse implication holds. In this direction, Mistretta \cite[Lemma 2.2]{mis} showed that linear (semi)stability of $( L, V )$ is in fact equivalent to (semi)stability of $\mlv$ when $\deg ( L ) \ge 2g+2c$ and $V \subseteq H^0 ( C, L )$ is a general subspace of codimension $c \leq g$. Using this result, he showed in \cite[Theorems 2.7 and 2.8]{mis} that for a general subspace $V \subseteq H^0 ( C, L )$ of codimension $c \leq g$, the bundle $\mlv$ is semistable.

 \vspace{.25cm}
 
\noindent Subsequently, Mistretta and Stoppino \cite[Theorem 1.1]{mistrettastoppino} gave conditions for the equivalence between the (semi)stability of $\mlv$ and linear (semi)stability of a linear series $( L, V )$ with $\deg ( L ) \leq 2 \cdot \dim ( V ) - 2 + \Cliff (C)$, and proved the semistability of $M_L$ when $L$ computes the Clifford index or $\deg ( L ) \geq 2g - \Cliff ( C )$. 
 However, Mistretta and Stoppino \cite[{\S} 8]{mistrettastoppino}, and subsequently the first author, E.\ Mistretta and H.\ Torres-L\'opez \cite{CMT}, exhibited various counterexamples showing that in general linear stability is weaker than slope stability of $\mlv$, even for complete coherent systems.

 \vspace{.25cm}
 
\noindent The goal of the present article is twofold. Firstly, after preliminaries in {\S} \ref{Preliminaries} on coherent systems and the \emph{Butler diagram} introduced in \cite{CT}, we give in {\S} \ref{RankOne} an overview of existing results relating linear stability and Butler's conjecture for coherent systems of rank one. In {\S} \ref{Hyperell}, we construct new examples showing in some sense that the situation studied in \cite{CMT} is not an isolated phenomenon. We note some natural questions complementing those in \cite[{\S} 5]{CMT}.

 \vspace{.25cm}
 
\noindent Our second object is to define linear semistability for coherent systems $(E, V)$ of higher rank. This is done in terms of a slope inequality for generated coherent subsystems of $(E, V)$, directly generalising Definition \ref{LinStRankOne}, and has not to our knowledge been formulated in this way before. As in the rank one case, (semi)stability of $\mev$ a priori implies linear (semi)stability of $( E, V )$, but not conversely. Thus it is natural to ask again under which conditions the converse implication does hold.

 \vspace{.25cm}
 
\noindent Now a thorough study of the semistability of dual span bundles associated to coherent systems of higher rank was initiated in \cite{bmno}. In particular, coherent systems $( E, V )$ are studied for which $\mev$ is not semistable. To get a feel for linear semistability in higher rank, and to place our investigation in context, we show in these cases the slightly stronger statement that $(E, V)$ is not linearly semistable.

 \vspace{.25cm}
 
\noindent Next, we give in Proposition \ref{Sufficient2d4} a sufficient condition under which linear stability of a type $(2, d, 4)$ coherent system $( E, V )$ implies slope stability of $\mev$. This uses the invariant $d_3$ from the gonality sequence of $C$ described in \cite[{\S} 4]{langenewstead}. We use this to prove a case of Conjecture \ref{StrongButlerConjecture} over a bielliptic curve.

 \vspace{.25cm}
 
\noindent In the other direction: In Theorem \ref{counterex} we construct, over any curve, a linearly stable coherent system $( E, V )$ with $\mev$ unstable. As in \cite[{\S} 8]{mistrettastoppino}, this is obtained as the pullback of a coherent system over $\PP^1$. This shows that semistability of $\mev$ is in general stronger than linear stability of $( E, V )$ also in higher rank.

 \vspace{.25cm}
 
\noindent We conclude by giving some open questions and directions of current and future research on linear stability for generated coherent systems of higher rank.

\subsection*{Acknowledgements} The first author is supported by project IN100723, ``Curvas, Sistemas lineales en superficies proyectivas y fibrados vectoriales'' from DGAPA, UNAM. He acknowledges Casa Matematica Oaxaca (CMO) and to the Organizing Committee for the support and invitation to participate as speaker at the BIRS--CMO workshop ``Moduli, Motives and Bundles -- New Trends in Algebraic Geometry'' which was held at the CMO in Oaxaca, Mexico from September 18 to September 23, 2022. The first and second authors also thank the research group ``Task Design in Mathematics Education'' of Oslo Metropolitan University for financial support. The second author thanks the Centro de Ciencias Matem\'aticas of UNAM Morelia for financial support, hospitality and excellent working conditions. The third author is supported by a doctoral fellowship from Conahcyt, Mexico.

\section{Preliminaries on coherent systems and Butler diagrams} \label{Preliminaries}

\noindent Let $C$ be a complex projective smooth curve of genus $g \ge 0$. Here we recall some definitions and fundamental facts about generated coherent systems and their dual span bundles.

\subsection{Generated coherent systems} \label{BackgroundCohSys}

\noindent For more detail about coherent systems, see for example \cite{bgpmn} and \cite{newsteadcohsys}.

 \vspace{.25cm}
 
\noindent A \textsl{coherent system of type $(r, d, n)$} over $C$ is a pair $(E, V)$ where $E \to C$ is a vector bundle of rank $r$ and degree $d$, and $V \subseteq H^0 ( C, E )$ is a subspace of dimension $n$. The coherent system is said to be \textsl{non-complete} if $V \subset H^0 ( C, E )$ is a proper subspace, and \textsl{complete} when $V = H^0 ( C, E )$. A \textsl{coherent subsystem} of $( E, V )$ is a coherent system $( F, W )$ where $F \subseteq E$ is a subbundle and $W \subseteq V \cap H^0 ( C, F )$. If $W = V \cap H^0 ( C, F )$ then $(F, W)$ is said to be a \textsl{complete} subsystem of $(E, V)$.

 \vspace{.25cm}
 
\noindent If $E$ is a coherent system and $W \subseteq H^0 ( C, E )$ a vector subspace, we set
\begin{equation} E_W \ := \ \Image  \left( \ev \colon W \otimes \Oc \ \to \ E \right) , \label{EW} \end{equation}
the subsheaf generated by $W$. The sheaf $E_W$ is locally free since $C$ has dimension one, but it is not saturated in general.

\begin{definition} \label{alphaSt}
Let $\alpha$ be a positive real number. The \textsl{$\alpha$-slope} of a coherent system $( E, V )$ is defined by
\[ \mu_\alpha ( E, V ) \ := \ \frac{\deg ( E ) + \alpha \cdot \dim ( V )}{\rank ( E )} . \]
Then $(E, V)$ is said to be \textsl{$\alpha$-semistable} if $\mu_\alpha ( F, W ) \le \mu_\alpha ( E, V )$ for all proper nonzero coherent subsystems $( F, W )$, and \textsl{$\alpha$-stable} if inequality is strict for all such $( F, W )$.
\end{definition}

\noindent King and Newstead \cite{kingnewstead} constructed moduli spaces for coherent systems (called ``Brill--Noether pairs'') using GIT:

\begin{theorem}[Theorem 1 of \cite{kingnewstead}] \label{moduli}
For each $\alpha > 0$, there exists a projective coarse moduli scheme $G_\alpha ( r, d, n )$ for S-equivalence classes of $\alpha$-semistable coherent systems of type $(r, d, n)$.
\end{theorem}

\noindent This shows in particular that the notion of ``generality'' makes sense for coherent systems. Using the fact that $C$ is complete, it is not hard to check that generatedness is an open property in families of coherent systems. Furthermore, suppose that $\alpha$ is close to zero. Then it follows from the discussion on \cite[p.\ 3]{but} that
\begin{itemize}
\item if $E$ is a stable vector bundle, then $(E, V)$ is $\alpha$-stable;
\item if $E$ is not semistable, then $(E, V)$ is not $\alpha$-semistable; and
\item if $(E, V)$ is $\alpha$-semistable, then $E$ is semistable.
\end{itemize}

\noindent In particular, this shows why Conjecture \ref{StrongButlerConjecture} is stronger than Conjecture \ref{ButlerConjecture} for general curves.

\subsection{Linear semistability, Butler diagrams and subbundles of dual span bundles}

\noindent Let $( L, V )$ be a generated linear series over $C$. We recall the following equivalent formulation \cite[Definition 2.2]{CMT} of linear (semi)stability of the image of $C \to \PP V^\vee$.

\begin{definition} \label{LinStRankOne}
Let $( L, V )$ be a generated coherent system of type $(1, d, n)$ over a curve $C$; equivalently, a base point free $g^{n-1}_d$ on $C$. Then $( L, V )$ is \textsl{linearly semistable} if for any linear subspace $W \subset V$ of dimension $w \ge 2$, we have
\begin{eqnarray*}
\frac{\deg (L_W)}{w-1} \ \geq \ \frac{\deg(L)}{n - 1}
\end{eqnarray*}
where $L_W$ is the invertible subsheaf generated by $W$ as in (\ref{EW}). If inequality is strict for all such $W$, then $( L, V )$ is \textsl{linearly stable}.
\end{definition}

\begin{remark} \label{SlopeEquivRankOne}
Let $(L, V)$ be as above, and consider the dual span bundle $\mlv$ defined as in (\ref{DefnDSB}). For any generated subsystem $( L', W )$ of $( L, V )$, there exists a commutative diagram 
\begin{equation*}
  \xymatrix{0  \ar[r]  & M_{W,L'}   \ar[r] \ar@{^{}->}[d]      &  W \otimes \Oc  \ar[r] \ar@{^{}->}[d] & L'   \ar[r] \ar[d]  & 0\\
				0\ar[r]	& \mlv \ar[r]_{}& V \otimes \Oc  \ar[r]   & L \ar[r] & 0  .  }
				\end{equation*}
As noted in \cite[Remark 3.2]{mistrettastoppino}, the condition of being linearly stable is equivalent to the bundle $\mlv$ not being destabilized by the subbundle $M_{W, L'}$ for any generated subseries $( L', W )$ of $( L, V )$.
\end{remark}

\noindent The above may be generalized as follows. Given a subbundle $S \subseteq \mlv$, as observed in \cite[{\S} 1]{butler} and \cite[{\S} 2]{CT} there exist a subspace $W \subseteq V$ and a bundle $F_S$ fitting into the following diagram:

\begin{equation}\label{butlerdiagram}
 \xymatrix{0  \ar[r]  & S  \ar[r] \ar@{^{(}->}[d]  &  W \otimes \Oc \ar[r] \ar@{^{(}->}[d] & F_S \ar[r] \ar[d]^{\alpha}  & 0 \\
0 \ar[r]	& \mlv \ar[r] & V \otimes \Oc  \ar[r]   & L \ar[r] & 0  . }
\end{equation}

\noindent Indeed, we define $W \hookrightarrow V$ by $W^\vee := \Image \left( V^{\vee}\stackrel{\phi}{\rightarrow} H^0 ( C, S^\vee ) \right)$. Note that $W^{\vee}$ generates $S^{\vee}$ because $V^\vee$ generates $\mev^\vee$. 
 The bundle $F_S$ can be defined by
\[
F_S^{\vee} \ := \ \Ker \left( W^{\vee} \otimes \Oc \ \rightarrow \ S^\vee \right) \ = \ M_{W^\vee , S^\vee} .
\]
We call (\ref{butlerdiagram}) the \textsl{Butler diagram of $( L, V )$ by $S$}. When $V = H^0 ( C, L )$, we refer simply to the \textsl{Butler diagram of $L$ by $S$}.

\vskip2mm

\noindent 
By \cite[{\S} 1]{butler}, the following properties hold in (\ref{butlerdiagram}).
		
		\begin{enumerate}
		\renewcommand{\labelenumi}{(\alph{enumi})}
	
      \item $W$ is a subspace of $H^0 ( C, F_S )$.
		
		\vspace{.05cm}
		
			\item The bundle $F_S$ is generated by $W$, and $h^0 ( C, F_S^\vee ) = 0$.
		
		\vspace{.05cm}
		
			\item The induced morphism $\alpha \colon F_S \rightarrow L$ is not zero.
			
			\vspace{.05cm}
			
			\item Let $S \subset \mlv$ be a subbundle of maximal slope. Then $\deg (F_S) \leq \deg (I)$, where $I = \Image (\alpha)$.
			Moreover, if $S$ is a destabilizing bundle of $\mlv$, then $\rank (F_S) = 1$ if and only if $\deg (F_S) = \deg (I)$.

		\end{enumerate}

\begin{remark}
The Butler diagram is a particular case of a more general construction: Given a generated coherent system $( E, V )$ of type $(r, d, n)$ over $C$, where $r \ge 1$, the dual span bundle $\mev = \Ker \left( V \otimes \Oc \to E \right)$ gives rise to a Butler diagram defined in the same way as for rank one above (see for instance \cite{but}). The above properties (a)--(c) are still valid, and property (d) must be modified as follows. Let $S \subset \mev$ be a subbundle of maximal slope. Then $f \leq \deg (I)$ where $I := \Image (\alpha)$. Moreover, if $S$ is a destabilizing bundle of $\mev$, then $\rank (F_S) = \rank(I)$ if and only if $f = \deg (I)$.
\end{remark}

\noindent As we shall see, the Butler diagram plays an important role in connecting linear stability of $(E, V)$ with slope stability of $\mev$.

\section{Review of results on linear stability for generated linear series} \label{RankOne}

\noindent Let $C$ be an irreducible projective smooth complex curve of genus $g \geq 2$, and let $L \in Pic^d(C)$ be a globally generated line bundle. Consider a generating subspace $V\subseteq H^0 ( C, L )$ of dimension $n$. The following conjecture concerns linear stability of the pair $( L, V )$.
\vskip2mm

\begin{conjecture}[Conjecture 6.1 of \cite{mistrettastoppino}]
    Let $( L, V )$ be a generated linear series as above. If $\deg (L) - 2 ( \dim (V) - 1 ) \leq \Cliff (C)$, then linear (semi)stability of $( L, V )$ is equivalent to (semi)stability of $\mlv$.
    \label{Conj1}
\end{conjecture}
This conjecture was proven under certain conditions:
\vskip2mm
\begin{theorem}[Theorem 6.3 of \cite{mistrettastoppino}]
    Conjecture \ref{Conj1} holds in the following cases:
    \begin{itemize}
        \item $H^0( C, L ) = V$
        \item $\deg (L) \leq 2g - \Cliff (C) + 1$
        \item $V \ne H^0 ( C, L)$ and $\codim_{H^0(C, L)} (V) < h^1 ( C, L ) + g/(\dim (V) - 2)$
        \item $\deg (L) \geq 2g$ and $\codim_{H^0 ( C, L )} (V) \leq (\deg (L) - 2g)/2$.
    \end{itemize}
\end{theorem}

Also, in section 8 of \cite{mistrettastoppino}, the authors show that there exists a counterexample to Conjecture \ref{Conj1} on any curve for non-complete linear series:
\vskip2mm
\begin{theorem}[Proposition 8.4 of \cite{mistrettastoppino}]
    On any curve $C$ there exists a non-complete linear system $( L, V )$ which is linearly stable and such that the vector bundle $\mlv$ is unstable.
\end{theorem}
\vskip2mm

\noindent In the proof of this result, authors produce a linearly stable $( L, V )$ satisfying
$$ \deg (L) \ > \ \gamma \cdot ( \dim (V) - 1 ) $$
where $\gamma$ is the gonality of $C$, and where $\mlv$ is unstable. Consequently, they formulate the following two conjectures for complete and non-complete linear series.
\vskip2mm

\begin{conjecture}[Conjecture 8.6 of \cite{mistrettastoppino}] \label{Conj2}
    Let $( L, V )$ be a base point free non-complete linear series. If $\deg (L) \leq \gamma \cdot ( \dim (V) - 1 )$, then linear (semi)stability of $( L, V )$ is equivalent to (semi)stability of the vector bundle $\mlv$.
\end{conjecture}

\begin{conjecture}[Conjecture 8.7 of \cite{mistrettastoppino}] \label{Conj3}
For any curve $C$ and any line bundle $L$ on $C$, linear (semi)stability of $( L, H^0( C, L) )$ is equivalent to (semi)stability of the vector bundle $M_L$.
\end{conjecture}

For the complete case, in \cite{CT} the authors dualize the Butler diagram of $L$ by $S$ to obtain a multiplication map 
$$ m_W \colon W \otimes H^0( C, \Kc ) \ \rightarrow \ H^0( C, S^\vee \otimes \Kc ) . $$

\noindent They prove that $m_W$ is surjective under some conditions and use this surjectivity on general curves to prove that Conjecture \ref{Conj3} is true under certain assumptions. Moreover, they gave conditions characterizing when $M_L$ fails to be stable for general curves:
\vskip2mm
\begin{theorem}[Corollary 4.1 of \cite{CT}]
    Let $L$ be a globally generated line bundle over a general Petri curve $C$ of genus $g>1$. Suppose that $h^0( C, L) = n$. Then
    \begin{itemize}
        \item Linear (semi)stability of $L$ is equivalent to (semi)stability of $M_L$.
        \item The vector bundle $M_L$ fails to be stable if and only if the following three conditions hold:
        \begin{enumerate}
            \item $h^1 ( C, L ) = 0$.
            \item $\deg (L) = g + n - 1$ and $n$ divides $g$.
            \item There is an effective divisor $Z$ with $h^0 ( C, L(-Z) ) = h^0( C, L ) - 1$ and $\deg (Z) = 1 + g/(n-1)$.
        \end{enumerate}
    \end{itemize}
\end{theorem}
\vskip2mm

\noindent For certain complete linear series on special curves, in \cite{CMT} the authors prove the following:
\vskip2mm 
\begin{theorem}[Theorem 4.1 of \cite{CMT}] \label{PlaneSeptic}
    On any smooth plane curve $C$ of degree $7$, a general element of $W^2_{15} (C)$ satisfies: 
    \begin{itemize}
        \item The complete linear series $(L, H^0 ( C, L ) )$ is base point free and linearly stable.
        \item The vector bundle $M_L$ is not semistable.        
    \end{itemize}
\end{theorem}
\noindent This result gives a counterexample to Conjecture \ref{Conj3} on plane curves of degree $7$. This has the following consequence:
\vskip2mm
\begin{theorem}[(Corollary 5.3 of \cite{CMT}] Let $C$ be a smooth plane curve of degree 7, and let $B_C(r,d,n)$ be the moduli space of semistable vector bundles $E$ on $C$ with rank $r$, degree $d$, and such that $h^0(C,E)\geq n$. Then all components of $B_C(2,15,3)$ consist of nongenerated bundles.
\end{theorem}

\vskip2mm

\begin{proof}
Suppose for a contradiction that $E \in B_C(2,15,3)$ were generated. Then there is a vector subspace $V \subset H^0 ( C, E )$ of dimension 3 that generates $E$. Then we have the following exact sequence
$$ 0 \ \to \ E^\vee \ \to \ V^\vee \otimes \Oc \ \to \ L \ \to \ 0 $$
where $L = \det (E)$ lies in a component of $W^2_{15}(C)$. As having semistable syzygy bundle is an open property, a general line bundle in that component would be base point free and have semistable syzygy bundle. But this is a contradiction to Theorem \ref{PlaneSeptic}.
\end{proof}

\noindent This concludes our overview of existing results relating linear stability to Butler's conjecture. In the next section we will give a construction showing that counterexamples similar to that studied in Theorem \ref{PlaneSeptic} exist over certain curves of any genus $g \ge 7$.

\section{Linearly stable systems with nonsemistable DSB over hyperelliptic curves} \label{Hyperell}

\noindent Let $C$ be a hyperelliptic curve of genus $g \ge 7$. Write $H$ for the hyperelliptic line bundle, and $| H | = g^1_2$ the hyperelliptic linear system. Then the canonical linear system $| \Kc |$ is $(g-1) \cdot g^1_2$. Let $n$ be an integer satisfying
\begin{equation} \label{NumericalHypothesis}
2 \ \le \ n \ \quad \text{ and } \quad 3n + 1 \ \leq \ g - 1.
\end{equation}
Then by \cite[Chap.\ I.2]{ACGH} we have $h^0 ( C, H^{2n+1} ) = 2n + 2$. If $V \subset H^0 ( C, H^{2n+1} )$ is a general subspace of dimension three, then we may assume that $V$ generates $H^{2n+1}$, and consider the dual span bundle
$$0 \ \to \ \mfv \ \to \ \Oc \otimes V \ \to \ H^{2n+1} \ \to \ 0 .$$
\vskip2mm
\noindent Note that $\mfv$ has rank $2$ and degree $-(4n+2)$, whence $\mu( \mfv ) = -2n-1$. Tensoring the above sequence by $H^n$ and taking cohomology, we obtain the multiplication map of sections
$$\mu_{V, H^{2n+1}} \colon V \otimes H^0(C, H^n ) \ \to \ H^0 (C, H^{3n+1}) .$$
Now $3n+1 \le g - 1$ by (\ref{NumericalHypothesis}). Thus as above we have
$$\dim (V) \cdot h^0 ( C, H^n ) \ = \ 3 ( n + 1 ) \quad \hbox{and} \quad h^0 (C, H^{3n+1}) = 3n + 2 .$$
Therefore, $\Ker ( \mu_{V, H^{2n+1}} ) = H^0(C,M_{V,H^{2n+1}} \otimes H^n)$ is nonzero, and there is a sheaf injection $H^{-n} \to \mfv$. It follows that $\mfv$ has a line subbundle of degree at least
$$ \mu(H^{-n}) \ = \ -2n \ > \ -2n - 1 \ = \ \mu (\mfv) . $$
In particular, it follows that
\begin{equation} \label{NotSemistable}
\hbox{the dual span bundle $\mfv$ is not semistable.}
\end{equation}

\noindent On the other hand, we have:

\begin{lemma} \label{HoweverLinSt}
The generated coherent system $(H^{2n+1} , V)$ over $C$ is linearly stable.
\end{lemma}

\begin{proof}
As $C$ is hyperelliptic, the canonical map $C \to |\Kc|^\vee$ factorizes via the double covering $\pi \colon C \to |H|^\vee$. In particular, every section of $\Kc$ is invariant under the hyperelliptic involution, and the natural injection $\pi^* \colon H^0 ( \PP^1 , \opo ( g - 1 ) ) \hookrightarrow H^0 ( C, \Kc )$ is an isomorphism.
\vskip2mm

\noindent As $H^{2n+1} = \pi^* \opo (2n+1)$, by (\ref{NumericalHypothesis}) we may choose a sheaf injection $H^{2n+1} \to \Kc$, and then $\pi^* \colon H^0 ( \PP^1 , \opo ( 2n+1 ) ) \to H^0 ( C, H^{2n+1} )$ is also an isomorphism. Therefore, $V = \pi^* \bV$ for a uniquely determined $\bV \subseteq H^0 ( \PP^1 , \opo (2n+1) )$, and $( H^{2n+1} , V ) = \left( \pi^* \opo(2n+1), \pi^* \bV \right)$.
\vskip2mm

\noindent Now since $2n + 1 \ge 5$ by (\ref{NumericalHypothesis}) and $V$ is general in $\Gr ( 3 , H^0 ( C, H^{2n+1} )$, by \cite[Proposition 8.2]{mistrettastoppino} the image $\PP^1$ in $\PP V^\vee$ is linearly stable; that is, $( \opo ( 2n+1 ) , \bV )$ is a linearly stable generated coherent system over $\PP^1$. By \cite[Remark 7.8]{mistrettastoppino} or Lemma \ref{PullbackLinSt} below, $( H^{2n+1} , V )$ is linearly stable.
\end{proof}

\noindent Summing up, by (\ref{NotSemistable}) and Lemma \ref{HoweverLinSt} we have:

\begin{theorem} \label{HyperellipticCounterexample} Let $C$ be any hyperelliptic curve of genus $g \ge 7$. Then for each $n$ satisfying (\ref{NumericalHypothesis}) there exists a linearly stable coherent system of type $( 1 , 4n+2 , 2 )$ over $C$ whose associated dual span bundle is desemistabilized by a subsheaf of the form $H^{\otimes -n}$. \end{theorem}

\begin{remark}
As mentioned in the previous section, in \cite[Theorem 4.1]{CMT} a similar construction gives a counterexample to the implication ``$(L, V)$ linearly stable implies $\mlv$ slope stable'' for a special curve of genus $15$. Theorem \ref{HyperellipticCounterexample} shows that this kind of situation is not an isolated phenomenon, but can arise over special curves of any genus $g \ge 7$.
\end{remark}

\subsection{Questions} \label{QuestionsRankOne}

The final section of \cite{CMT} contains a list of interesting open questions related to linear stability and its relation to Butler's conjecture and Brill--Noether theory. We mention some other interrelated open questions.

\begin{question} Find further examples of linearly (semi)stable $(L, V)$ for which $\mlv$ is not (semi)stable. \end{question}

\begin{question} Find necessary and sufficient conditions under which linear stability of $(L, V)$ is equivalent to slope stability of $\mlv$.  In particular, in both \cite[Theorem 4.1]{CMT} and Theorem \ref{HyperellipticCounterexample} the curve is Brill--Noether special. What can be said for general curves? \end{question}

\section{Linear stability for higher rank coherent systems}

\noindent For the remainder of the article, we discuss the notion of linear (semi)stability in the context of higher rank bundles over the curve $C$. A priori, linear stability is a property of a variety $X$ equipped with a morphism $\psi \colon X \to \PP^{n-1}$ (cf.\ Definition \ref{MumfordDefn}). When $X$ is the curve $C$, we have used the equivalent Definition \ref{LinStRankOne} from \cite{CMT}, characterizing linear stability in terms of the cohomology of the linear system $(L, V)$ over $C$. Taking a cue from this, we shall define linear stability in the present work directly as a property of generated coherent systems. This turns out to be convenient for applications to Butler's conjecture (but see Remark \ref{OtherDefn}).

\subsection{Defining linear stability in higher rank}

\begin{definition} \label{LStHigherRank}
Let $(E, V)$ be a generated coherent system of type $(r, d, n)$ over $C$. We say that $(E, V)$ is \textsl{linearly semistable} if for all subspaces $W \subset V$ such that $E_W \not\cong W \otimes \Oc$, we have
\begin{equation} \label{LStIneq} 
\frac{\deg (E_W)}{\dim (W) - \rank (E_W)} \ \ge \ \frac{d}{n - r} ,
\end{equation}
where $E_W$ is the subsheaf generated by $W$ as in (\ref{EW}). If inequality is strict for all such $W$, we say that $(E, V)$ is \textsl{linearly stable}.
\end{definition}

\noindent The following equivalent formulation will be useful.

\begin{lemma} \label{LStHigherRankEquiv} 
Let $(E, V)$ be a generated linear system of type $(r, d, n)$ over $C$. Then $(E, V)$ is linearly semistable if and only if for all generated coherent subsystems $(F, W)$ of $(E, V)$ where $\deg (F) > 0$, we have
\begin{equation} \label{LStIneqEquiv}
\frac{\deg (F)}{\dim (W) - \rank (F)} \ \ge \ \frac{d}{n - r} .
\end{equation}
Moreover, $(E, V)$ is linearly stable if and only if inequality is strict for all such $(F, W)$.
\end{lemma}

\begin{proof}
Suppose that $(E, V)$ is linearly semistable, and let $(F, W)$ be a generated subsystem with $\deg (F) > 0$. As $W \subseteq H^0 ( C, F)$ and $( F, W )$ is generated,
\[
F \ = \ \Image \left( W \otimes \Oc \ \to \ E \right) \ = \ E_W
\]
by definition (\ref{EW}). As $\deg (F) > 0$, moreover, $F$ is not a trivial bundle. Then (\ref{LStIneq}) implies the desired inequality (\ref{LStIneqEquiv}).

\noindent For the converse: Suppose $W \subset V$ is such that $E_W$ is not trivial. Then $( E_W, W )$ is a generated subsystem of positive degree in $E$, and (\ref{LStIneqEquiv}) implies (\ref{LStIneq}) as desired.
\end{proof}

\begin{remark} \label{InitialRemarks}
\quad
\begin{enumerate}
\renewcommand{\labelenumi}{(\alph{enumi})}
\item Note that in contrast to the definition \ref{alphaSt} of $\alpha$-stability of $(E, V)$, the condition (\ref{LStIneqEquiv}) is nontrivial when $F$ is a full rank subsheaf (elementary transformation) of $E$. In Definition \ref{LinStRankOne}, in fact every subsheaf considered is of this form.
\item As in the rank one case, clearly it is sufficient to require (\ref{LStIneqEquiv}) for \emph{complete} linear subsystems (cf.\ {\S} \ref{BackgroundCohSys}).
\item For $r = 1$, the condition $\dim (W) \ge 2$ in Definition \ref{LinStRankOne} implies that $\deg (L_W) > 0$. However, for $r \ge 2$ this no longer holds: any generated $(E, V)$ admits generated subsystems of the form $\left( \Oc^{\oplus m} , H^0 ( C, \Oc )^{\oplus m} \right)$ for $2 \le m \le r$. If (\ref{LStIneqEquiv}) is written
\[
(n - r) \cdot \deg F \ \ge \ d \cdot ( \dim W - \rank F )
\]
then any such trivial subsystem would prevent $( E, V )$ from being linearly stable. This is why we require (\ref{LStIneq}) only for subspaces $W$ generating subsheaves of positive degree; equivalently, subsheaves of rank strictly less than $\dim (W)$. See also Remark \ref{NonzeroAnalogy}.
\end{enumerate}
\end{remark}

\noindent As before, we wish to study the relation between linear (semi)stability of $(E, V)$ and slope (semi)stability of the dual span bundle $\mev$ defined in (\ref{DefnDSB}). The key to this is the following generalization of Remark \ref{SlopeEquivRankOne}. Note firstly that if $(F, W) \subset (E, V)$ is an inclusion of generated coherent systems (where possibly $\rank (F) = \rank (E)$), there is a Butler diagram
\[
\xymatrix{ 0 \ar[r] & \mfw \ar[d] \ar[r] & \Oc \otimes W \ar[d] \ar[r] & F \ar[d] \ar[r] & 0 \\
 0 \ar[r] & \mev \ar[r] & \Oc \otimes V \ar[r] & E \ar[r] & 0 . }
\]

\begin{lemma} \label{SlopeCrit}
Suppose that $(E, V)$ is a generated coherent system of type $(r, d, n)$ over $C$. Then $(E, V)$ is linearly semistable if and only if the slope inequality $\mu ( \mfw ) \le \mu ( \mev )$ is satisfied for all generated coherent subsystems $(F, W)$ of $(E, V)$ where $\deg (F) > 0$ and $W \subseteq V \cap H^0 ( C, F )$ generates $F$. The system $(E, V)$ is linearly stable if and only if inequality is strict for all such $(F, W)$.
\end{lemma}

\begin{proof}
For any generated subsystem $(F, W)$ of $(E, V)$ with $\deg (F) > 0$, we have
\[
\frac{\deg ( F )}{\dim ( W ) - \rank ( F )} \ = \ \mu ( \mfw^\vee ) .
\]
Thus the statement follows from Lemma \ref{LStHigherRankEquiv}. \end{proof}

\begin{remark} \label{NonzeroAnalogy}
We observe that the dual span bundle of a coherent system of the form $\left( \Oc \otimes \C^w , H^0 ( C, \Oc \otimes \C^w ) \right)$ is zero. Thus restricting to subsystems of positive degree in Definition \ref{LStHigherRank} and Lemma \ref{LStHigherRankEquiv} is analogous to requiring the slope inequality only for \emph{nonzero} dual span subbundles of $\mev$ in Lemma \ref{SlopeCrit}.
\end{remark}

\noindent In particular, as in the case $r = 1$ discussed earlier, if $\mev$ is slope (semi)stable then $(E, V)$ is linearly (semi)stable. In what follows, we shall discuss the reverse implication for $r \ge 2$.

\subsection{Some examples}

\noindent In \cite{bmno}, the relation between $\alpha$-stability of a generated $( E, V )$ (Definition \ref{alphaSt}) and slope stability of the corresponding $\mev$ is studied for $r \ge 2$ and $\alpha$ close to zero. In particular, several situations are considered (sometimes implicitly) where a coherent system $( E, V )$ has nonsemistable $\mev$. We shall use Lemma \ref{SlopeCrit} in the following cases to show the slightly stronger statement that $( E, V )$ is not linearly semistable.

 \vspace{.25cm}
 
\noindent Notice that neither Definition \ref{LStHigherRank} nor Definition \ref{alphaSt} requires that the ambient bundle $E$ be semistable. The role of semistability of $E$ is discussed below.

\begin{example} \label{ExaOne}
We recall the construction of \cite[Example 3.7]{bmno}. Set $g = 2$ and let $r$ and $a$ be positive integers. Let $d$ be an integer satisfying $3n + 2a < d < 4n + 2a$, and set $m = d - 2n - a$. Let $V \subseteq H^0 ( C, E )$ be a subspace of dimension $r + m$. By \cite[Theorem 3.3]{bmno}, there exists a generated coherent system $(E, V)$ of type $(r, d, r+m)$ which is $\alpha$-stable for $\alpha$ close to zero. Then $\mev$ has slope $-\frac{d}{m} = -\frac{d}{d - 2r - 2a} < -2$. The contrapositive of the argument of \cite{bmno} shows that $h^0 ( C , \Hom ( \Kc^\vee , \mev ) ) \ne 0$, so $\Kc^\vee$ is a destabilizing subsheaf of $\mev$.

\noindent Let us now show that $(E, V)$ is not linearly semistable. Firstly, we claim that $\Kc^\vee$ is saturated in $\mev$. Let $\Kc^\vee ( D )$ be the saturation, where $D$ is an effective divisor. Since $\mev^\vee$ is generated by $V^\vee$, so is the quotient bundle $\Kc ( - D )$. Hence $D$ has degree $0$, $1$ or $2$. If $2$ then $\Kc (-D)$ is trivial. But then $\Oc$ is a subbundle of $\mev$, which is impossible since $h^0 ( C , \mev ) = 0$. 
 If $\deg D = 1$, then $\Kc ( - D )$ is a generated degree $1$ line bundle, which is impossible since $g = 2$. Thus $D = 0$ and $\Kc^\vee$ is a subbundle of $\mev$. Moreover, the image of the map $V^\vee \to \mev^\vee \to \Kc$ is isomorphic to $H^0 ( C, \Kc )$. It follows that there exists a diagram of vector bundles
\[
\xymatrix{ \Kc^\vee \ar[r]^-\varepsilon \ar[d] & H^0 ( C ,  \Kc )^\vee \otimes \Oc \ar[r] \ar[d] & \Kc \ar[d] \\
 \mev \ar[r] & V \otimes \Oc \ar[r] & E . }
\]
(Note that $\varepsilon$ is identified with the transposed evaluation map of $\Kc$. The fact that the quotient is isomorphic to $\Kc$ is then a consequence of the fact that $g = 2$.) Therefore, $\Kc$ is a subbundle of $E$ which is generated by $H^0 ( C, \Kc ) \cap V$. As $g = 2$, we have
\[
\frac{\deg ( \Kc )}{\dim ( V \cap H^0 ( C, \Kc ) ) - \rank ( \Kc )} \ = \ \frac{2}{2 - 1} \ = \ 2 .
\]
On the other hand,
\[
\frac{\deg E}{\dim V - r} \ = \ \frac{d}{m} \ = \ \frac{d}{d-2r-a} .
\]
As by hypothesis $d < 4r + 2a$, we have $2d - 4r - 2a < d$, whence $\frac{d}{d-2r-a} > 2$. Thus by Lemma \ref{SlopeCrit}, the system $(E, V)$ is not linearly semistable. \end{example}

\begin{remark}
We note in passing that in another situation, \cite[Theorem 1.2]{but}, a condition for stability of $M_E$ is that $\Kc$ should not inject into $E$. The role of the property $\Kc \subset E$ remains to be fully explored.
\end{remark}

\noindent As the coherent system $(E, V)$ of Example \ref{ExaOne} is $\alpha$-stable for $\alpha$ close to zero, in particular $E$ is a semistable vector bundle (cf.\ {\S} \ref{BackgroundCohSys}). However, the above argument does not directly use the semistability of $E$. In contrast, we now discuss two other situations from \cite{bmno} where the stability properties of $E$ are significant. The next example is taken from the proof of \cite[Theorem 3.11 (i)]{bmno}.

\begin{example} \label{ExaTwo}
Suppose $d > 2rg - r$ and that $E$ is a strictly semistable bundle of rank $r$ and degree $d$. Let $E' \subset E$ be a subbundle of rank $r'$ and degree $d'$ with $\frac{d'}{r'} = \frac{d}{r}$. Then $E'$ is also semistable, and $h^1 ( C, E ) = h^1 ( C, E' ) = 0$. Thus, using Riemann--Roch, we have
\[
\frac{d'}{h^0 ( C, E' ) - r'} \ = \ \frac{d'}{d' - r'(g-1) - r'} \ = \ \frac{d'/r'}{d'/r' - g} \ = \ \frac{d/r}{d/r - g} \ = \ \frac{d}{d - rg} \ = \ \frac{\deg (E)}{h^0 ( C, E) - r} .
\]
Thus, by Lemma \ref{SlopeCrit}, at best $(E, V)$ is linearly strictly semistable.
\end{example}

\noindent The next example is adapted from the proof of \cite[Lemma 4.4]{bmno}.

\begin{example} \label{ExaThree}
Let $(E, V)$ be a generated coherent system of type $(r, d, n)$ where $E$ is semistable. Suppose that $(G, W)$ is a generated subsystem of type $(s, e, m)$ with $e > 0$, and thus also $m > s$. Now suppose that $\frac{s}{m-s} \le \frac{r}{n-r}$; equivalently,
\begin{equation} \label{one}
\frac{s}{r(m - s)} \ \le \ \frac{1}{n - r} .
\end{equation}
As $E$ is semistable, we have
\begin{equation} \label{two}
e \ \le \ \frac{ds}{r} .
\end{equation}
Thus we obtain
\[
\mu ( M_{G, W}^\vee ) \ = \ \frac{e}{m - s} \ \le \ \frac{ds}{r(m-s)} \le \ \frac{d}{n-r} \ = \ \mu ( \mev^\vee ) ,
\]
where the first and second inequalities follow from (\ref{two}) and (\ref{one}) respectively. By Lemma \ref{SlopeCrit}, the system $(E, V)$ is at best linearly semistable. If either (\ref{one}) or (\ref{two}) is strict (for example, if $E$ is actually stable), then $\mev$ is not semistable.
\end{example}

\begin{question}
A similar argument shows that the coherent system $(E, V)$ in the proof of \cite[Theorem 4.1]{bmno}, which is $\alpha$-nonstable for some $\alpha$, also fails to be linearly semistable. In view of this and Remark \ref{InitialRemarks} (a), the connection between $\alpha$-stability and linear stability would be natural to investigate.
\end{question}

\noindent In light of the above examples, in the context of Butler's conjecture it is natural to ask in the case $r \ge 2$ for conditions under which linear (semi)stability of $(E, V)$ implies slope (semi)stability of $\mev$. It is also of interest give counterexamples such as those in \cite[{\S} 8]{mistrettastoppino} and \cite{CMT} where the implication does not hold. We shall now offer results in both of these directions.

\subsection{A sufficient condition for stability of rank two DSBs}

Let $C$ be a curve. Before proceeding, we recall from \cite[{\S} 4]{langenewstead} that the \textsl{gonality sequence} $d_1 , d_2 , \ldots$ of $C$ is defined by
\[
d_k \ := \ \min \{ \deg (L) : \hbox{ $L$ a line bundle over $C$ with } h^0 ( C, L ) \ge k + 1 \} .
\]

\noindent Now we can give a straightforward condition on a coherent system of type $(2, d, 4)$ under which linear stability implies slope stability of the dual span bundle.

\begin{proposition} \label{Sufficient2d4}
Let $(E, V)$ be a generated coherent system of type $(2, d, 4)$. If $(E, V)$ is linearly stable and $d < 2 \cdot d_3$, then $\mev$ is stable.
\end{proposition}

\begin{proof} Suppose $\mev$ is not stable. Let $S \subset \mev$ be a destabilizing line subbundle. The full Butler diagram of $(E, V)$ by $S$ is as follows:
\[
\xymatrix{ & 0 \ar[d] & 0 \ar[d] & N \ar[d] & \\
 0 \ar[r] & S \ar[r] \ar[d] & \Oc \otimes W \ar[r] \ar[d] & F_S \ar[r] \ar[d] & 0 \\
 0 \ar[r] & \mev \ar[r] \ar[d] & \Oc \otimes V \ar[r] \ar[d] & E \ar[r] \ar[d] & 0 \\
 0 \ar[r] & Q \ar[r] \ar[d] & \Oc \otimes \left( \frac{V}{W} \right) \ar[r] \ar[d] & T \ar[r] \ar[d] & 0 \\
 & 0 & 0 & 0 . & }
\]
Firstly, suppose that $N = 0$, so that $F_S$ is a subsheaf of $E$ of degree $-\deg(S) \le \frac{d}{2}$. As $F_S$ is generated by $W$, we have $\dim (W) - \rank (F_S) \ge 1$. Therefore
\[ \frac{\deg (F_S)}{\dim (W) - \rank (F_S)} \ \le \ \frac{d/2}{1} \ = \ \frac{d}{4-2} \ = \ \frac{\deg (E)}{\dim (V) - \rank (E)} , \]
and by Lemma \ref{SlopeCrit} we see that $(E, V)$ is not linearly stable.

 \vspace{.25cm}
 
\noindent Suppose that $N$ is nonzero. By the snake lemma, we have an exact sequence
\[
0 \ \to \ N \ \to \ Q \ \to \Oc \otimes \left( \frac{V}{W} \right) \ \to \ T \ \to \ 0 .
\]
Thus $N$ injects into $Q$ and $Q/N$ injects into a trivial sheaf. As $Q$ has rank one and $\Oc$ is torsion free, $Q \cong N$ and $\Oc \otimes \left( \frac{V}{W} \right) \cong T$.

 \vspace{.25cm}
 
\noindent If $T$ is nonzero then $E$ has a trivial quotient. Thus $E$ is an extension $0 \to L \to E \to \Oc \to 0$ where $L$ is a line bundle of degree $d$ and $(L, H^0 ( C, L ) \cap V )$ a generated type $(1, d, 3)$ subsystem of $(E, V)$. Then
\[
\frac{\deg (L)}{\dim \left( H^0 ( C, L ) \cap V \right) - \rank (L)} \ = \ \frac{d}{2} \ = \ \frac{\deg(E)}{\dim (V) - \rank (E)} ,
\]
so $(E, V)$ is not linearly stable by Lemma \ref{SlopeCrit}.

 \vspace{.25cm}
 
\noindent If $T = 0$ then $W = V$ and $h^0 ( C, S^\vee ) \ge 4$, whence $\deg ( S^\vee ) \ge d_3$. As $S$ destabilizes $\mev$, we have
\[
d_3 \ \le \ \deg ( S^\vee ) \ \le \ \frac{d}{2} ,
\]
and $d \ge 2 d_3$.

 \vspace{.25cm}
 
\noindent In summary, we have shown that if $\mev$ is not stable, then $(E, V)$ is not linearly stable and/or $d \ge 2 d_3$. The statement follows.
\end{proof}

\begin{remark}
The above should be compared with \cite[Theorem 5.3 and Corollary 5.4]{mistrettastoppino}, which for systems of type $(1, d, n)$ give sufficient conditions for stability of $\mlv$ in terms of linear stability of $(L, V)$ and the Clifford index of $C$.
\end{remark}

\noindent Before proceeding, we prove a slight generalization of part of \cite[Remark 7.8]{mistrettastoppino} on the behavior of linear stability under pullbacks.

\begin{lemma} \label{PullbackLinSt}
Let $f \colon \tC \to C$ be a degree $k$ covering of curves. Let $( E, V )$ be a generated coherent system on $C$. If $(E, V)$ is linearly (semi)stable then so is $( f^* E , f^* V )$.
\end{lemma}

\begin{proof} Let $( \tF, \tW )$ be a generated coherent subsystem of $( f^* E, f^* V )$. As $f^* \colon H^0 ( C, E ) \to H^0 ( \tC , f^* E )$ is injective, $\tW = f^* W$ for a subspace $W \subseteq V$. Now there is a commutative diagram
\[
\xymatrix{ \cO_{\tC} \otimes f^* W \ar[dr] \ar[drr]^a \ar[ddr]_b & & \\
& f^* E_W \ar[r] \ar[d] & E_W \ar[d] \\
& \tC \ar[r]^f & C }
\]
where $a( \tx , f^* s ) = s( f(x) )$ and $b(\tx , f^* s ) = \tx$, and the existence of $c$ follows from the universal property of fiber products. But by definition, $\left( f^* E \right)_{f^* W}$ is the image of the evaluation map $\cO_{\tC} \otimes f^* W \to f^* E$. Thus $\left( f^* E \right)_{f^* W} = f^* E_W$.

 \vspace{.25cm}
 
\noindent Therefore, we obtain
\[
\frac{\deg ( \tF )}{\dim ( \tW ) - \rank ( \tF )} \ = \ \frac{k \cdot \deg ( E_W )}{\dim ( W ) - \rank ( E_W )} .
\]
Thus if $( E, V )$ is linearly semistable then
\[
\frac{\deg ( \tF )}{\dim \tW - \rank ( \tF )} \ \ge \ k \cdot \frac{\deg (E)}{\dim ( V ) - \rank ( E )} \ = \ \frac{\deg ( f^* E )}{\dim ( f^* V ) - \rank ( f^* E )} .
\]
It follows that $( f^* E, f^* V )$ is linearly semistable. The statement for linear stability is proven by changing ``$\ge$'' to ``$>$'' in the last inequality.
\end{proof}

\subsubsection{An example}

\noindent We shall now use Proposition \ref{Sufficient2d4} and the above lemma to prove a particular case of Conjecture \ref{StrongButlerConjecture}. The following, involving a bielliptic curve, complements \cite[Theorem 5.10 (iii)]{bmno} which applies to systems of larger degree over general curves.

Let $Z$ be an elliptic curve. Let $M_1$ and $M_2$ be line bundles over $Z$ of degrees $3$ and $4$ respectively. Then by Serre duality (with $g = 1$) and Riemann--Roch,
\[
h^1 ( Z, M_2^\vee \otimes M_1 ) \ = \ h^0 ( Z, M_1^\vee \otimes M_2 ) \ = \ 1 .
\]
Thus there is, up to isomorphism of bundles, a unique indecomposable extension
\begin{equation} \label{ConstructionE}
0 \ \to \ M_1 \ \to \ E \ \to \ M_2 \ \to \ 0 .
\end{equation}
As $h^1 ( Z, M_1 ) = 0 = h^1 ( Z, M_2 )$, we have $h^0 ( Z, E ) = 7$ by Riemann--Roch.

\begin{proposition} \label{EllipticLinSt}
For a general $V \in \Gr \left( 4, H^0 ( Z, E ) \right)$, the type $(2, 7, 4)$ coherent system $( E, V )$ is generated and linearly stable.
\end{proposition}

\begin{proof}
Let us firstly check that $E$ is generated by global sections. For any $z \in Z$, by Serre duality and since $K_Z = \Oz$ we have $h^1 ( Z, E ( -z ) ) = h^0 ( Z, E^\vee (z) )$. Dualizing (\ref{ConstructionE}) and tensoring by $\Oz (z)$ we see easily that $h^0 ( Z, E^\vee (z) ) = 0$. 
 Thus by Riemann--Roch, $h^0 ( Z, E (-z) ) = 5 = h^0 ( Z, E ) - \rank (E)$ for any $z \in Z$, and $E$ is generated. It follows that a general coherent system $(E, V)$ where $V \subset H^0 ( Z, E )$ is of dimension $4$ is generated.

 \vspace{.25cm}
 
\noindent For linear stability, we must check (\ref{LStIneq}) for each $W \subset V$ such that $E_W$ is not trivial. We check the cases where $W$ has dimension $1$, $2$ and $3$ separately. Firstly, if $\dim W = 1$ then $E_W \cong \Oc$, and there is nothing to check.

 \vspace{.25cm}
 
\noindent Suppose that $\dim W = 2$. We wish to show that $E_W$ has rank two. If not, then
\[
W \ \subseteq \ i \left( H^0 ( Z , M ) \right) \cap V
\]
where $M$ is a line bundle of degree $e > 0$ and $i \colon M \to E$ a vector bundle injection. (Note that such a $W$ generates an invertible subsheaf of positive degree whose saturation is $M$, but for the present argument we do not need to assume that $E_W = M$.) As $E$ is indecomposable, $e \le 3$, and since $h^0 ( Z, M ) \ge \dim (W) = 2$ we have $e \ge 2$. For such an $i$, we have a short exact sequence
\[
0 \ \to \ M \ \xrightarrow{i} \ E \ \to \ E/M \ \to \ 0 ,
\]
and we obtain a saturated element of $\Quot^{1, 7-e} ( E )$. The obstruction space of the latter scheme at the point $\left[ M \xrightarrow{i} E \right]$ is $H^1 ( Z , \Hom ( M , E/M ) )$. By Serre duality, this is identified with $H^0 ( Z , (E/M)^\vee \otimes M )$. As $\deg ((E/M)^\vee \otimes M ) = 2e - 7 < 0$, this is zero. Therefore
\[
\dim \Quot^{1,7-e} (E) \ = \ \chi \left( Z , M^\vee \otimes (E/M) \right) \ = \ 7 - 2e .
\]
Now for a fixed $M \to Z$ of degree $e \in \{2, 3 \}$ and a fixed $i \colon M \to E$, we obtain a subspace $i \left( H^0 ( Z, M ) \right)$ of $H^0 ( Z, E )$. By Riemann--Roch, $\dim \left( i \left( H^0 ( Z, M ) \right) \right) = e$. We denote this space by $\C^e_i$. Consider now the locus
\[
\Sigma_i \ := \ \left\{ V \in \Gr \left( 4, H^0 ( Z , E ) \right) : \dim \left( V \cap \C^e_i \right) \ge 2 \right\} .
\]
Any $V \in \Sigma_i$ fits into a sequence $0 \to V_1 \to V \to V_2 \to 0$ where $V_1 \in \Gr \left( 2, \C^e_i \right)$ and $V_2 \in \Gr \left( 2, H^0 ( Z , E ) / V_1 \right)$. Therefore,
\[
\dim ( \Sigma_i ) \ \le \ 2 \cdot ( e - 2 ) + 2 \cdot \left( \left( h^0 ( Z, E ) - 2 \right) - 2 \right) \ = \ 2e - 4 + 2 \cdot ( 5 - 2 ) \ = \ 2e + 2 .
\]
Thus the dimension of the union
\begin{equation} \label{UnionSigmai}
\bigcup_{\left[ M \xrightarrow{i} E \right] \ \in \ \Quot^{1, 7 - e} (E)} \Sigma_i
\end{equation}
is at most
\[
\dim ( \Sigma_i ) + \dim \Quot^{1, 7 - e} (E) \ \le \ ( 2e + 2 ) + ( 7 - 2e ) \ = \ 9 .
\]
On the other hand, $\dim \left( \Gr \left( 4, H^0 ( Z, E ) \right) \right) 
 = 12$. Thus for $e \in \{ 2 , 3 \}$, a general $V \in \Gr \left( 4, H^0 ( Z, E ) \right)$ does not intersect $\C^e_i$ in dimension two or more for any $i \colon M \to E$. We may therefore assume that $E_W$ is a trivial bundle of rank two, as desired.

 \vspace{.25cm}
 
\noindent Finally, we consider $W$ of dimension three. Assuming $V$ to be outside (\ref{UnionSigmai}), we may assume that $\rank ( E_W ) = 2$. In order to satisfy (\ref{LStIneq}), we must show that for general $V \in \Gr \left (4, H^0 ( Z, E ) \right)$, for any elementary transformation $F \subset E$ with $\dim ( V \cap H^0 ( Z, F ) ) = 3$ where $F$ is generated by $V \cap H^0 ( Z, F )$, we have
\[
\frac{\deg (F)}{3 - 2} \ > \ \frac{\deg ( E )}{ 4 - 2 } \ = \ \frac{7}{2} ,
\]
that is, $\deg F \ge 4$. It will suffice to show that a general $V$ does not contain a three-dimensional subspace of $H^0 ( Z, F )$ for any generated elementary transformation $F \subset E$ of degree $e \le 3$. 

 \vspace{.25cm}
 
\noindent \textbf{Claim:} A generated bundle $F \to Z$ of rank $2$ and degree $e \le 3$ with $h^0 ( Z, F ) \ge 3$ must be of one of the following forms:
\begin{enumerate}
\item[(i)] $F \cong \Oz \oplus M'$ for a line bundle $M'$ of degree $e \in \{ 2, 3 \}$
\item[(ii)] a nonsplit extension $0 \to N_1 \to F \to N_2 \to 0$ for line bundles $N_1$ and $N_2$ of degree $1$ and $2$ respectively; in particular, $\deg (F) = 3$
\end{enumerate}
To see this: If $F$ is decomposable then it is a sum $N_1 \oplus N_2$ of generated line bundles whose degrees sum to $e$. As no degree $1$ line bundle over a curve of positive genus is generated, the only possibility is that in (i). If $F$ is indecomposable, then it is a nonsplit extension $0 \to N_1 \to F \to N_2 \to 0$ where $\deg ( N_1 ) + \deg ( N_2 ) = e$ and $\deg ( N_2 ) \le \deg (N_1 ) + 1$ and $h^0 ( Z, N_1 ) + h^0 ( Z, N_2 ) \ge 3$. It is easy to see that the only possibility is that $\deg (N_1) = 1$ and $\deg (N_2) = 2$, and in particular $e = 3$. Thus $F$ is of the form (ii). This proves the Claim.\\
\\
Now if $j \colon \Oz \oplus M' \to E$ is a sheaf injection and $V$ intersects $j ( H^0 ( Z, \Oz \oplus M' ) ) = \C^{e+1}$ in dimension $3$, then in particular $V \cap H^0 ( Z, M' )$ has dimension at least $2$ for the degree $e$ line bundle $M'$. We have seen above that this does not happen for a generic $V$. Thus we may assume that $F$ is of form (ii). In particular, $e = 3$, and since $h^1 ( Z, N_1 ) = h^1 ( Z, N_2 ) = 0$, by Riemann--Roch we have $h^0 ( Z, F ) = 3$. 

 \vspace{.25cm}
 
\noindent Now any degree $3$ elementary transformation $F \subset E$ defines an element of $\Quot^{0, 4} ( E )$, which is of dimension $\rank (E) \cdot 4 = 8$. For a fixed $\left[ F \xrightarrow{j} E \right] \in \Quot^{0, 4} (E)$, consider the locus
\[
\Sigma_j \ := \ \left\{ V \in \Gr \left( 4, H^0 ( Z, E ) \right) : H^0 ( Z, F ) \subset V \right\} .
\]
Any $V$ belonging to $\Sigma_j$ fits into a sequence of the form
\[
0 \ \to \ j \left( H^0 ( Z, F ) \right) \ \to \ V \ \to \ V' \to \ 0
\]
where $V' \in \Gr \left( 1, H^0 ( Z , E ) / W \right) = \PP^3$. Therefore, the dimension of the union
\[
\bigcup_{\left[ F \to E \right] \ \in \ \Quot^{0, 4} (E)} \Sigma_j
\]
is bounded above by
\[
\dim ( \Sigma_j ) + \dim \left( \Quot^{0, 4} ( E ) \right) \ \le 3 + 8 \ = \ 11 .
\]
Again, this is strictly less than $\dim \Gr \left( 4, H^0 ( Z , E ) \right) = 12$. Thus if $V \in \Gr \left( 4, H^0 ( Z , E ) \right)$ is general, then $V$ does not intersect $H^0 ( Z, F )$ in dimension $3$ for any subsheaf $F \subset E$ of rank $2$ and degree $\le 3$.

 \vspace{.25cm}
 
\noindent Putting these facts together, we conclude that $( E, V )$ is linearly stable, as desired.
\end{proof}

\begin{corollary} \label{biellipticStable}
Let $Z$ and $E$ be as above, and let $V \in \Gr \left( 4, H^0 ( Z, E ) \right)$ be general. Let $f \colon C \to Z$ be a double cover, where $C$ has genus $g \ge 6$. Then $M_{f^* V, f^* E}$ is stable.
\end{corollary}

\begin{proof}
By Proposition \ref{EllipticLinSt}, the coherent system $(E, V)$ is linearly stable. Thus $( f^* E , f^* V )$ is linearly stable by Lemma \ref{PullbackLinSt} and is of type $(2, 14, 4)$. By \cite[Remark 4.5 (d)]{langenewstead}, for a bielliptic curve $C$ of genus $g \ge 6$, we have $d_3 = 2 \cdot 3 + 2 = 8$. Thus $\deg ( f^* E ) = 14 < 2 \cdot d_3$. By Proposition \ref{Sufficient2d4}, the dual span bundle $M_{f^* V , f^* E}$ is stable.
\end{proof}

\noindent This implies Conjecture \ref{StrongButlerConjecture} for generated coherent systems of type $(2, 14, 4)$ over bielliptic curves of sufficiently large genus:

\begin{corollary} \label{biellipticButler}
For any bielliptic curve $C$ of genus $g \ge 6$, Conjecture \ref{StrongButlerConjecture} holds for a component of the moduli space of $\alpha$-stable type $(2, 14, 4)$ coherent systems on $C$, for $\alpha$ close to $0$.
\end{corollary}

\begin{proof}
Let $(f^* E, f^* V)$ be as above. By the proof of \cite[Proposition 5.3]{langenarasimhan}, the vector bundle $f^* E$ has Segre invariant $2$, and in particular is stable. By the discussion following Theorem \ref{moduli}, the system $(f^* E, f^* V)$ is $\alpha$-stable for $\alpha$ close to $0$. 
 By Corollary \ref{biellipticStable}, the dual span bundle $M_{f^* V, f^* E}$ is a stable vector bundle. As stability of the DSB is an open condition, $\mev$ is stable for a general $( E, V )$ in the component of $G ( \alpha ; 2, 14, 4 )$ containing $(f^* E, f^* V)$.
\end{proof}

\begin{remark} \label{rmkBielliptic}
If $E$ is a generated bundle of rank $2$ and degree $14$ over a curve $C$, then in view of the sequence $0 \to \Oc \to E \to \det (E) \to 0$ we see that $\det(E)$ is a line bundle of degree $14$ satisfying $h^0 ( C, \det (E) ) \ge 3$. A standard Brill--Noether argument shows that for $g \ge 18$, a general curve of genus $g$ admits no such line bundle, and hence no such $E$ either. Thus one can formulate Conjecture \ref{StrongButlerConjecture} in situations where Butler's original conjecture does not arise.
\end{remark}

\subsection{A linearly stable coherent system with unstable DSB}

To contrast with the previous subsection, we now exhibit, over any curve, a coherent system $(E, V)$ of type $(2, d, 4)$ which is linearly stable but such that $\mev$ is not semistable. Following \cite[{\S} 8]{mistrettastoppino}, we shall construct a linearly stable coherent system over $\PP^1$ and pull it back via a map $C \to \PP^1$. (Our method for proving linear stability, however, is different from that in \cite{mistrettastoppino}; it is the same as that in Proposition \ref{EllipticLinSt}.)

\begin{theorem} \label{counterex}
Fix an integer $e \ge 3$, and let $E \to \PP^1$ be the bundle $\opo (e) \oplus \opo (e+1)$. Let $V \subset H^0 ( \PP^1 , E )$ be a general subspace of dimension four. Then $(E, V)$ is generated and linearly stable, but $\mev$ is not semistable.
\end{theorem}

\begin{proof}
By Riemann--Roch, we see that $h^0 ( \PP^1 , E ) = 2e + 3$. It is easy to see that a general subspace $V \subset H^0 ( \PP^1 , E )$ of dimension four generates $E$. For any such $V$, the bundle $\mev$ has rank two and odd degree over $\PP^1$. Thus it cannot be semistable. We claim that $( E, V )$ is nonetheless linearly stable. This follows from an argument which is virtually identical to that of Proposition \ref{EllipticLinSt}, but in fact is technically simpler. We must verify (\ref{LStIneq}) for each $W \subset V$ such that $E_W$ is not trivial. We check again the cases where $W$ has dimension $1$, $2$ and $3$ separately. Firstly, if $\dim W = 1$ then $E_W \cong \Oc$, and there is nothing to prove. For the case when $\dim W = 2$, we wish to show that $\rank ( E_W ) = 2$. If this is not the case, then
\[
W \ \subseteq \ j \left( H^0 ( \PP^1 , \opo (t) ) \right) \cap V
\]
where $j \colon \opo (t) \hookrightarrow E$ is a vector bundle injection and $1 \le t \le e + 1$. Such a $j$ defines a saturated element of $\Quot^{1, 2e+1-t} ( E )$. It is easy to check that the obstruction space of the latter scheme at the point $\left[ \opo (t) \xrightarrow{j} E \right]$ is zero, whence
\[
\dim \Quot^{1,2e+1-t} (E) \ = \ \chi \left( \PP^1 , \opo ( 2e + 1 - 2t ) \right) \ = \ 2 ( e - t - 1 ) .
\]
Now for a fixed $t \ge 1$ and a fixed $j \colon \opo(t) \to E$, we obtain a subspace $j \left( H^0 ( \PP^1 , \opo ( t ) ) \right)$ of $H^0 ( \PP^1 , E )$ of dimension $t+1$. We denote this space by 
$\C^{t+1}_j$. Consider now the locus
\[
\Sigma_j \ := \ \left\{ V \in \Gr \left( 4, H^0 ( \PP^1 , E ) \right) : \dim V \cap \C^{t+1}_j \ge 2 \right\} .
\]
As in the proof of Proposition \ref{EllipticLinSt}, one shows that
 $\dim \Sigma_j \le 2 ( 2e + t - 4 )$, and that the dimension of the union
\begin{equation} \label{UnionSigmaj}
\bigcup_{\left[ \opo ( t ) \xrightarrow{j} E \right] \ \in \ \Quot^{1, 2e+1-t} (E)} \Sigma_j
\end{equation}
is at most $\dim ( \Sigma_j ) + \dim \left( \Quot^{1, 2e+1-t} (E) \right) \le 6e - 6$. On the other hand,
$$\dim \Gr \left( 4, H^0 ( \PP^1 , E ) \right) \ = \ 4 ( 2e + 3 - 4 ) \ = \ 8e - 4, $$
which exceeds $6e-6$ since $e \ge 3$.
 Thus for $1 \le t \le e+1$, a general $V \in \Gr \left( 4, H^0 ( \PP^1 , E ) \right)$ does not intersect $\C^{t+1}_j$ in dimension two or more for any $j \colon \opo (t) \to E$. We may therefore assume that $E_W$ is a trivial bundle of rank two, as desired.

\vspace{.25cm}
 
\noindent 
Finally, we consider $W$ of dimension three. Assuming $V$ to be outside (\ref{UnionSigmaj}), we may assume that $\rank ( E_W ) = 2$. 
In order to satisfy (\ref{LStIneq}), we wish to show that when $V$ is general, for any $W \subset V$ of dimension three we have
\[
\frac{\deg (E_W)}{3 - 2} \ > \ \frac{2e + 1}{ 4 - 2 } ;
\]
that is, $\deg (E_W) \ge e + 1$. In other words, we must show that a general $V$ does not contain a three-dimensional subspace of $H^0 ( \PP^1 , F )$ for any elementary transformation $F \subset E$ of degree at most $e$ generated by $V \cap H^0 (C, F )$.

\vspace{.25cm}
 
\noindent Now we observe that if $F' \subset E$ is an elementary transformation of degree $e' < e$ which is generated by $V \cap H^0 ( C, F' )$, then $F'$ is contained in at least one elementary transformation $F \subset E$ of degree $e$; and in this case, $F$ is generically generated by $V \cap H^0 ( C, F' )$. Thus it suffices to prove that a general $V$ does not contain a three-dimensional subspace of $H^0 ( \PP^1 , F )$ for any \emph{generically} generated elementary transformation $F \subset E$ of degree \emph{exactly} $e$.

\vspace{.25cm}
 
\noindent Such an $F$ must be of the form $\opo ( a_1 ) \oplus \opo ( a_2 )$ where $a_1 + a_2 = e$. As moreover $F$ is generically generated, $0 \le a_1 \le a_2 \le e$, and so $h^0 ( \PP^1 , F ) = ( a_1 + 1 ) + ( a_2 + 1 ) = e + 2$.

\vspace{.25cm}
 
\noindent Now any degree $e$ elementary transformation $F \subset E$ defines an element of $\Quot^{0, e + 1} ( E )$, which is (possibly reducible) of dimension $2 \cdot ( e + 1 )$. For a fixed such $F \subset E$, consider the locus
\[
\Sigma_F \ := \ \left\{ V \in \Gr \left( 4, H^0 ( \PP^1 , E ) \right) : W \subset V \hbox{ for some } W \in \Gr \left( 3, H^0 ( \PP^1 , F ) \right) \right\} .
\]
Any $V$ belonging to $\Sigma_F$ fits into a sequence of the form
\[
0 \ \to \ W \ \to \ V \ \to \ V' \to \ 0
\]
for some $W \in \Gr \left( 3, H^0 ( \PP^1 , F ) \right)$ and
\[
V' \ \in \ \Gr \left( 1, H^0 ( \PP^1 , E ) / W \right) \ = \ \PP \left( H^0 ( \PP^1 , E ) / W \right) \ = \ \PP^{2e-1} .
\]
On shows that the dimension of the union
\[
\bigcup_{\left[ F \to E \right] \ \in \ \Quot^{0, e+1} (E)} \Sigma_F
\]
is bounded above by
\[
\dim \Sigma_F + \dim \Quot^{0, e+1} ( E ) \ \le \ (5e - 4) + 2 (e+1) \ = \ 
 7e - 2< \dim \Gr \left( 4, H^0 ( \PP^1 , E ) \right) = 8e - 4,
\]
 for $e\geq 3$. Thus a general $V \in \Gr \left( 4, H^0 ( \PP^1 , E ) \right)$ does not intersect $H^0 ( \PP^1 , F )$ in dimension three for any subsheaf $F \subset E$ of rank two and degree $\le e$. Putting these facts together, we conclude that $( E, V )$ is linearly stable, as desired.

\vspace{.25cm}
 
\noindent Now let $C$ be any curve of genus $g \ge 0$. Let $f \colon C \to \PP^1$ be a nonconstant map. By Lemma \ref{PullbackLinSt}, the coherent system $( f^* E, f^* V )$ is linearly stable, while $M_{f^* V, f^* E} \cong f^* \mev$ is not semistable.
\end{proof}

\subsection{Remarks and questions}

\begin{remark} \label{OtherDefn}
In Definition \ref{LStHigherRank}, linear semistability is defined as a condition on linear systems $(E, V)$ over $C$. The geometric meaning of this condition is perhaps less apparent than the condition in Mumford \cite[Definition 2.16]{mumfordstability}. Noting that $V$ is naturally isomorphic to a subspace of $H^0 ( \PP E^\vee , \opeo)$ one can show that Definition \ref{LStHigherRank} is in fact equivalent to linear semistability, in Mumford's sense, of the image of the natural map $\PP E^\vee \to \PP V^\vee$. In a forthcoming paper, we propose to explore more fully the geometric implications of linear semistability.
\end{remark}

\begin{question}
Find more examples analogous to \cite[Proposition 8.4]{mistrettastoppino}, \cite[Theorem 4.1]{CMT} and Theorem \ref{counterex} showing that linear (semi)stability of $(E, V)$ is in general weaker than slope stability of $\mev$. In particular, are there examples where $E$ is (semi)stable and/or $(E, V)$ is complete?
\end{question}

\begin{question}
Mistretta and Stoppino \cite[Proposition 3.3]{mistrettastoppino} use Clifford indices to prove linear stability of a large class of rank one coherent systems. This approach seems unlikely to transfer directly to higher rank, because Lange and Newstead's higher rank Clifford indices \cite{langenewstead} are defined for semistable bundles, and so are not always applicable in the corresponding way. Therefore, it is natural to ask for other methods, perhaps generalising \cite[Proposition 8.2]{mistrettastoppino}, for deciding when a coherent system $(E, V)$ is linearly (semi)stable.
\end{question}

\begin{question}
As mentioned before, Mumford proved \cite[Theorem 4.12]{mumfordstability} that a linearly (semi)stable curve in $\PP^n$ is Chow (semi)stable. Is this also true for projective bundles? If so, can the results of \cite{bptl} be generalised to Hilbert schemes parametrising such projective bundles?
\end{question}

\begin{question}
What is the relation, if any, between $\alpha$-semistability and linear semistability of a generated coherent system $(E, V)$?
\end{question}

\begin{question}
Two other properties investigated in \cite{mistrettastoppino} which are stronger than both linear stability of $(E, V)$ and slope stability of $\mev$, are cohomological stability of $\mev$ (see \cite{ein}) and the existence of a theta divisor for $\mev$ (see for example \cite{Popa}, \cite{montserrat} and \cite{CTtheta}). It would be interesting to determine other conditions, both for rank one and for higher rank, under which linear stability implied these properties also.
\end{question}

\end{document}